\documentclass[reqno]{amsart}
\usepackage{amsmath, amssymb, amsthm, geometry, enumerate, graphicx, amsfonts, hyperref, mathrsfs}
\hypersetup{colorlinks=true,linkcolor=red, anchorcolor=blue, citecolor=blue, urlcolor=red, filecolor=magenta, pdftoolbar=true}
\theoremstyle{plain}
\newtheorem{thm}{\it Theorem}[section]

\theoremstyle{remark}
\newtheorem{defn}[thm]{Def{}inition}
\newtheorem{rem}[thm]{Remark}
\newtheorem{exa}[thm]{Example}
\numberwithin{equation}{section}
\usepackage{enumerate}
\usepackage{relsize}
\allowdisplaybreaks
\begin{document}

\title [Frames Associated with Weyl-Heisenberg Group and Extended Affine Group]{Frames With Several Generators  Associated with Weyl-Heisenberg Group and Extended Affine Group}

\author[Divya Jindal]{Divya Jindal}
\address{{\bf{Divya Jindal}}, Department of Mathematics,
University of Delhi, Delhi-110007, India.}
\email{divyajindal193@gmail.com}

\author[Lalit   Kumar Vashisht]{Lalit  Kumar  Vashisht}
\address{{\bf{Lalit  Kumar  Vashisht}}, Department of Mathematics,
University of Delhi, Delhi-110007, India.}
\email{lalitkvashisht@gmail.com}

\begin{abstract}
We study the construction of  Gabor frames and wavelet frames for Weyl-Heisenberg group and extended affine group by using  contraction between the affine group and the Weyl-Heisenberg group  due to  Subag, Baruch, Birman and Mann.  Firstly, we give construction of  Gabor frames with several generators from a unitary irreducible representation associated to the  Weyl-Heisenberg group.  Wavelet frames  with several generators associated with an  extended affine group have been obtained. A relation between  frames for Weyl-Heisenberg group and extended affine group is also discussed. Finally, we show that  frames of Gabor and wavelet structure are stable under small perturbations.

\end{abstract}

\renewcommand{\thefootnote}{}

\footnote{
The research of Divya Jindal is supported by the Council of Scientific $\&$ Industrial Research (CSIR), India.  Grant No.: 09/045(1680)/2019-EMR-I. Lalit Kumar Vashisht is  supported by the Faculty Research Programme Grant-IoE, University of Delhi \ (Grant No.: Ref. No./IoE/2021/12/FRP).}

\maketitle

\baselineskip15pt
\section{Introduction}
Wavelet and Gabor analyses are perfect illustrations of the deep analogy between quantum physics and signal processing, see \cite{AG, FS}.  Gabor analysis yields a time frequency representation of signals while wavelet provide a time scale representation. More precisely, both Gabor analysis and wavelet analysis are integral part  of time-frequency analysis which gives uniform time and frequency resolution. In recent years, reconstruction of signals and its decomposition by using wavelet and Gabor systems have been extensively studied, see \cite{CK,  Classen, DabI,   FZUS,  KG, Heil20, H89, Hari, Zalik1, Zalik2} and many references therein.  The wavelet and Gabor analyses are both induced by coherent states  that are associated with a unitary irreducible representation of a Lie group, see \cite{AG, KG, H89, AP} for details. Wavelet analysis corresponds to unitary irreducible representations of the affine group while the Gabor analysis corresponds to that of the Weyl-Heisenberg group. Let $\mathcal{H}$ be a nonzero separable Hilbert space and $G$ be a locally compact group. A unitary representation of $G$ is a homomorphism $\pi$ from $G$ into the group $\mathcal{U} (\mathcal{H})$ of unitary operators on $\mathcal{H}$ that is continuous with respect to the strong operator topology, that is, a map $\pi: G\to \mathcal{U}(\mathcal{H})$ that satisfies $\pi(xy)=\pi(x)\pi(y)$ and $\pi(x^{-1})=\pi(x)^{-1}=\pi(x)^*$, and for which $x\to \pi(x)u$ is continuous from $G$ to $\mathcal{H}$ for any $u\in \mathcal{H}$.  Also, $\pi$ is said to be irreducible if $\pi$ admits only trivial (that is, $=\{0\}$ or $\mathcal{H}$) invariant subspaces. A Lie group can be obtained as a certain limit of another group with the help of group contraction, see \cite{SB} and references therein.  Let $G_0$ and $\widetilde{G_0}$ be two Lie groups. We say that $\widetilde{G_0}$ is a \emph{contraction} of $G_0$ if for every $\epsilon \in (0,1]$, there exists a diffeomorphism $F_{\epsilon}: \widetilde{G_0} \to G_0$ such that $F_{\epsilon}(e_0)=e_1$ where $e_0$, $e_1$ are the identity elements of $\widetilde{G_0}$ and  $G_0$, respectively, and $x.y=\lim_{\epsilon \ to 0^{+}}F_{\epsilon}^{-1}(F_{\epsilon}(x).F_{\epsilon}(y))$ for every $x,y \in \widetilde{G_0}$. Very recently, Subag,  Baruch, Birman, and  Mann, in \cite{SB}, studied relation between Gabor and wavelet analyses by using ``contraction'' between  affine group and Weyl-Heisenberg group. They obtained unitary irreducible representations of the Heisenberg group as contraction of representations of the extended affine group. Using this technique, Subag et al.,  in \cite{SB}, relates two analyses, where they contracted coherent states, tight frames and resolutions of the identity. In the direction of frame conditions,  Subag et al. constructed a family of time localized wavelet frames that contract to a Gabor frame. In \cite{SB}, Subag et al., also showed that how a family of frequency localized wavelet frames can be contracted to non-standard Gabor frames. They also gave deformation of Gabor frames to wavelet frames. To connect wavelet and Gabor systems by special type of groups, the authors of  \cite{DFR, BT1, BT2} gave some relations between the Gabor analysis and wavelet analysis by exhibiting both the affine group and the  Weyl-Heisenberg group  as subgroups of the four dimensional affine Weyl-Heisenberg group.

Motivated by above works which relates fundamental properties of frames from wavelet systems to Gabor systems via contraction, we study construction of  frames with several generators for Weyl-Heisenberg group and extended affine group.  Subag et al., in  \cite{SB},  gave tight frames with Gabor and wavelet structure associated with  Weyl-Heisenberg group and extended affine group. They obtain unitary irreducible representation of Weyl-Heisenberg group from unitary irreducible representation of extended affine group using group contraction. Using the representations given by Subag et al., in  \cite{SB},  we give existence of Gabor frames and wavelet frames with several generators associated with Weyl-Heisenberg group and extended affine group, respectively.
\subsection{Outline and contributions}
 The paper is structured as follows: To make the paper self-contained, Section \ref{sec2} gives a brief review on  Gabor systems, wavelet systems, Weyl-Heisenberg group and extended affine group. The main results are given in  Section \ref{sec3}. Our first result gives sufficient conditions for existence of  Gabor frames with several generators for the signal space $L^2(\mathbb{R}, dx)$ from a unitary irreducible representation associated to the  Weyl-Heisenberg group, see Theorem \ref{t1}. The conditions given in Theorem \ref{t1} are only sufficient are justified in Example \ref{eg1}. Theorem \ref{t3} gives wavelet frames with several generators for associated with an  extended affine group. A relation between frames for Weyl-Heisenberg group and extended affine group can be found in Theorem \ref{t4}. Theorem \ref{pthm1} of Section \ref{sec4} shows that  frames of both Gabor and wavelet structure are stable under small perturbations.
\section{Preliminaries}\label{sec2}
Symbols $\mathbb{N}$, $\mathbb{ Z}$,  $\mathbb{R}^{+}$,  $\mathbb{R}^{*}$ and $\mathbb{R}$ stand for a set of the natural numbers, integers, positive real numbers, non-zero real numbers and real numbers,  respectively. To make notations compatible with quantum physics, in this paper, a vector $\phi$ of a Hilbert space $\mathbb{H}$ is also denoted by $|\phi \rangle$ (ket vector or state in  $\mathbb{H}$). The inner product on $\mathbb{H}$ is denoted by $\langle \cdot | \cdot \rangle$. The inner product $\langle \cdot | \cdot \rangle$ is conjugate-linear in the first variable and linear in the first variable, that is,
\begin{align*}
\langle \phi| \alpha \psi + \beta \varphi \rangle = \alpha \langle \phi | \psi \rangle + \beta \langle \phi | \varphi \rangle  \ \   \text{and} \ \  \langle \alpha \phi| \psi \rangle = \overline{\alpha} \langle  \phi| \psi \rangle,
\end{align*}
for all vectors $|\phi \rangle$, $|\psi \rangle $ and $|\varphi \rangle$ in $\mathbb{H}$ and for all scalars $\alpha$, $\beta$.  Here, $\overline{\alpha}$ denote the complex-conjugate of the scalar $\alpha$. The norm on $\mathbb{H}$, denoted by $\|\cdot\|$, is defined as $\|\phi\| = \langle \phi | \phi\rangle$ for  $|\phi \rangle \in \mathbb{H}$. As is standard, the space of equivalence classes of square integrable (Lebesgue integrable) functions is denoted by $L^2(\mathbb{R}, dx)$. Here $dx$ is the Lebesgue measure. Engineers and physicists pronounced  $L^2(\mathbb{R}, dx)$ as the  ``\emph{signal space}''. We consider  the  following standard inner product on $L^2(\mathbb{R}, dx)$:
\begin{align}\label{inn1}
\langle \phi | \psi \rangle = \int_{\mathbb{R}} \overline{\phi(x)} \psi(x) dx, \  \ \phi, \,  \psi \in L^2(\mathbb{R}, dx).
\end{align}
The space $L^2(\mathbb{R}, dx)$  becomes  a Hilbert space with  respect to the  inner product defined in \eqref{inn1}.

To define Gabor and wavelet system, the following classes of operators, which acts unitarily on $L^2(\mathbb{R}, dx)$, are used:  Let  $a$, $b$  be positive real numbers,  $c$ be a non-zero real number and let $\phi$ be a non-zero vector  in $L^2(\mathbb{R})$.
\begin{align*}
&   \ T_a: \phi (x) \mapsto \phi (x-a) \ \quad (\text{Translation by} \ a );\\
& E_b: \phi (x) \mapsto e^{2\pi ib\cdot x} \phi(x) \quad (\text{Modulation by} \ b); \\
&  D_{c}: \phi (x) \mapsto  |c|^{\frac{1}{2}} \phi(cx) \ \ \,  \quad (\text{Dilation by} \  c ).
\end{align*}
A collection of vectors of the form $\mathcal{G}(a, b, \phi): = \{E_{mb}T_{na}\phi\}_{m, n \in \mathbb{Z}} = \{e^{2 \pi i m b x} \phi(x -a)\}_{m, n \in \mathbb{Z}}$ is called a \emph{Gabor system}; and the collection $\mathcal{W}(a, b, \phi): \{T_{k b a^j}D_{a^j}\phi\}_{j, k \in \mathbb{Z}} = \{a^{\frac{j}{2}} \phi(a^j x- k b)\}_{j, k \in \mathbb{Z}}$ is called a \emph{wavelet system}.

\subsection{Hilbert Frames} Dennis Gabor, in \cite{Gabor}, introduced a fundamental approach to signal decomposition in terms of elementary signal, for example, the Gaussian. Duffin and Schaeffer \cite{DS} while studying some intense problems related to non-harmonic Fourier series,  abstracted the fundamental notion of Gabor to introduce the notion of a frame  for the signal space $L^2((-\beta, \beta), dx)$. To be precise, they defined frames of exponentials for the space $L^2((-\beta, \beta), dx)$. Gabor frames or Weyl-Heisenberg frames originate in quantum mechanics by Neumann \cite{Neumann},  and in communication theory by the idea of   Gabor, see  \cite{Gabor}. A countable collection of vectors   $\{|w_k \rangle\}_{k \in \mathfrak{I}}$ in a separable Hilbert space $\mathbb{H}$  is a \emph{Hilbert frame } (or simply, \emph{frame}) for  $\mathbb{H}$ if  there exist positive real numbers $ \gamma_o$, $ \delta_o $ such that
\begin{align}\label{1.1}
\gamma_o \|w\|^2\leq  \sum_{k \in \mathfrak{I}} |\langle w, w_k\rangle|^2 \leq \delta_o \|w\|^2
\end{align}
holds for all $|w\rangle$ in   $\mathbb{H}$.
The scalars $ \gamma_o$, $ \delta_o $ are called lower and upper frame bounds of the frame $\{w_k\}_{k \in \mathfrak{I} }$. The frame  $\{|w_k\rangle\}_{k \in \mathfrak{I} }$  \emph{frame} is tight if it is possible to choose $\alpha_o =\delta_o$, and  \emph{normalized tight} (or\emph{Parseval},  if  $\alpha_o =\delta_o =1$. If only upper inequality holds in \eqref{1.1}, then we say that  $\{|w_k\rangle\}_{k \in \mathfrak{I}}$   is  a \emph{Bessel sequence} with \emph{Bessel bound} $\delta_o$. If $\{|w_k \rangle\}_{k \in \mathfrak{I}}$ is a frame for $\mathbb{H}$, then the map  $\mathrm{S}:  \mathcal{H} \rightarrow  \mathcal{H }$, known as the  \emph{frame operator}, given by
$\mathrm{S}: |w\rangle \mapsto \sum_{k \in \mathfrak{I}}\langle w, w_k \rangle w_k$
is linear, bounded and invertible on $\mathbb{H}$. Thus, a frame for $\mathbb{H}$ provides a decomposition, not necessarily unique,  for   every vector  $|w\rangle$ of $\mathbb{H}$:
\begin{align*}
|w\rangle = \mathrm{S}\mathrm{S}^{-1}|w\rangle =\sum_{k \in \mathfrak{I}} \big(\langle \mathrm{S}^{-1}w, |w_k\rangle\big) \,  |w_k\rangle.
\end{align*}
This decomposition turns out to be a powerful tool for deriving many fundamental results in abstract frame theory. Further, it plays central role  both in both pure mathematics and engineering science, for  technical details about this, we refer  to  excellent texts by Casazza and Kutyniok \cite{CK} and Heil \cite{Heil20}.
Nowadays the theory of  wavelets, distributed signal processing,  quantum mechanics, iterated function systems and many more topics are emerging in  important applications in frame theory, see \cite{DVII,  JVash, JV202021, VD3} and many references therein. A frame of the form $\mathcal{G}(a, b, \phi)$ is called the Gabor frame and wavelet frame is a frame of the form $\mathcal{W}(a, b, \phi)$.
\subsection{The Weyl-Heisenberg group}

The Weyl-Heisenberg group is the semidirect product $\mathbb{R}^2 \rtimes_\psi \mathbb{R}$, where $\psi:\mathbb{R} \rightarrow Aut(\mathbb{R}^2)$ is given by $\psi_\lambda((x_1,x_2))=(x_1+\lambda x_2,x_2)$ for every $\lambda,x_1,x_2 \in \mathbb{R}$;  and denote $\mathcal{W}=\mathbb{R}^2 \rtimes_\psi \mathbb{R}$ with multiplication law
\begin{align*}
((x_1,x_2),\lambda)((y_1,y_2),\nu)=(\psi_\lambda(y_1,y_2)+(x_1,x_2),\lambda+\nu).
\end{align*}
Recall that for every $P\in \mathbb{R}^*, Q \in \mathbb{R}$, the unitary irreducible representation of this group acting on $\mathcal{H}$, $\kappa^{P,Q}:\mathcal{W} \rightarrow \mathcal{U}(L^2(\mathbb{R},dx))$ is given by
\begin{align*}
 (\kappa^{P,Q}(((u_1,u_2),z))f)(x)= e^{i[P(u_1+xu_2)+Qu_2]}f(z+x).
\end{align*}

For $l \in \{1,2,\dots,N\}$, where $N$ is some fixed natural number, let $q_0^{(l)}$ and $p_0^{(l)}$ be real numbers such that $|q_0^{(l)}p_0^{(l)}|< 2\pi$. Consider the discrete subset of $\mathcal{W}$ given by
\begin{align*}
\mathcal{W}_{q_0^{(l)},p_0^{(l)}}=\bigg\{\bigg(\bigg(\frac{mnq_0^{(l)}p_0^{(l)}}{2},mp_0^{(l)}\bigg),nq_0^{(l)}\bigg)|n,m \in \mathbb{Z}\bigg\}.
\end{align*}
For $0\neq \phi_l \in L^2(\mathbb{R},dx)$ where $l \in \{1,2,\dots,N\}$, the sequence
\begin{align*}
\big|\phi_{(m,n,l)}^{P,Q}\big>=\kappa^{P,Q}\bigg(\bigg(\frac{mnq_0^{(l)}p_0^{(l)}}{2},mp_0^{(l)}\bigg),nq_0^{(l)}\bigg)\phi_l,
\end{align*}
where $n,m \in \mathbb{Z}$ is dense in $L^2(\mathbb{R},dx)$, see \cite{SB} for technical details.
\subsection{Extended Affine Group}\label{sec2.3}
The affine group $A$ realized as the upper half plane
\begin{center}
    $A=\{(a,b): a\in \mathbb{R}^+, b \in \mathbb{R} \}$,
\end{center}
and the multiplication law in $A$ is given by
\begin{center}
$(a_1, b_1)(a_2,b_2)=(a_1a_2, a_1b_2+b_1)$,
\end{center}
for all points $(a_1, b_1), (a_2, b_2)\in A$.

The support of a function $f$, denoted by  supp$f$, is  defined as  supp$f: =$ closure of the set $\{x: f(x) \ne 0\}$. Let $V$ be a subspace  of $L^2(\mathbb{R}, dx)$ given by
\begin{align}\label{eqTT1}
    V= \Big\{f \in L^2(\mathbb{R}, dx) : \mbox{supp}(\hat{f}) \subseteq [0, \infty)\Big\}
\end{align}
where $\hat{f}$ is the Fourier transform of $f$ given by
\begin{align*}
    \hat{f}(w)=\frac{1}{\sqrt{2\pi}}\int_{\mathbb{R}}f(x)e^{-iwx}dx.
\end{align*}
Recall that for any $\alpha \in \mathbb{R},$ the unitary irreducible representation of the affine group on the subspace $V$ of $L^2(\mathbb{R}, dx)$,  as given in \eqref{eqTT1},  is the map  $\eta^\alpha : A \to \mathcal{U}(V)$ given by
\begin{center}
 $(\eta^\alpha(a,b)f)(x)=\frac{1}{\sqrt{a}}f(\frac{x+\alpha b}{a}), \, x \in \mathbb{R}$;
\end{center}
 and correspondingly define the sequence
 \begin{center}
     $\phi^{\alpha}_{a,b}(x)=(\eta^\alpha(a,b)\psi)(x)=\frac{1}{\sqrt{a}}\psi(\frac{x+\alpha b}{a})$.
 \end{center}
The extended affine group, denoted by $EA$, is given by the direct product of the affine group $A$ with additive group of real numbers. That is,  $EA=A \oplus \mathbb{R}$. The multiplication  law on $EA$ is given by
\begin{center}
    $(a,b,y)(c,d,z)=(ac, ad+b, y+z)$.
\end{center}
For $\beta \in \mathbb{R},$ let $\chi_\beta(x)=e^{i\beta x}$ be unitary character of $\mathbb{R}.$ For every $\beta \in \mathbb{R}$, $\alpha \in \mathbb{R^*}$, the unitary irreducible representation of $EA$ acting on $V$ is the map  $\eta^\alpha \otimes \chi_\beta : EA \to \mu(V)$  given by
\begin{center}
    $(\eta^\alpha \otimes \chi_\beta(a,b,y))f)(x)=e^{i\beta y}\frac{1}{\sqrt{a}}f(\frac{x+\alpha b}{a})$.
\end{center}
For finding the relation between frames for $EA$ and frames for $W$, we need a different realization which is better suited. This realization is given by three linear invertible maps $S_1$, $S_2$, $S_3(\epsilon)$ defined by
\begin{enumerate}[$(i)$]
    \item $S_1: V \to L^2(\mathbb{R}^+, dx)$, \ $S_1(f)(x)=\hat{f}(x)|_{\mathbb{R}^+}$.

    \item $S_2: L^2(\mathbb{R}^+, dx) \to L^2(\mathbb{R}^+, \frac{1}{x}dx)$, \  $S_2(f)(x)=\sqrt{x}f(x)$.

    \item $S_3(\epsilon):  L^2(\mathbb{R}^+, \frac{1}{x}dx) \to L^2(\mathbb{R}^+, dx)$, \  $S_3(\epsilon)(f)(x)=f\circ\psi_\epsilon(x)$,
      \ where $\psi_\epsilon(x)= e^{-\epsilon x}$.
\end{enumerate}

Set $U_\epsilon = S_3(\epsilon)\circ S_2 \circ S_1$. For each $\epsilon \in (0, 1]$, $\alpha \in \mathbb{R}^*$ and $\beta \in \mathbb{R}$, interwine $\eta^\alpha \otimes \chi_\beta$ with $U_\epsilon$ to get the equivalent representation
\begin{center}
    $\zeta^{\beta, \alpha}_\epsilon : EA \to \mathcal{U}(V)$
\end{center}
\begin{align*}
    \begin{split}
(\zeta^{\beta, \alpha}_\epsilon(a,b,y)f)|_x & = U_\epsilon \circ (\eta^\alpha \otimes \chi_\beta)(a,b,y)\circ U_\epsilon)(f)(x)\\
& = e^{i\beta y}e^{i\alpha be^{-\epsilon x}}f(\frac{-\log a}{\epsilon}+x).
\end{split}
\end{align*}

\section {Main Results}\label{sec3}
We begin this section with the following definition.
\begin{defn}
Let $N$ be a fixed natural number. A collection of vectors $\{\big|\phi_{(m,n,l)}\big>\}_{m, n\in\mathbb{Z} \atop l\in\{1,2,\dots,N\}}$ in $\mathcal{H}$ is said to be a \emph{frame}  for $\mathcal{H}$ if there exist positive real constants $\alpha_o$ and  $\beta_o$ such that
\begin{align}
\alpha_o \|f\|^2 \leq \sum_{l = 1}^{N}\sum_{m, n\in \mathbb{Z}}|\langle f|\phi_{(m,n,l)}\rangle\|^2 \leq \beta_o\|f\|^2 \ \text{for all} \ f \in \mathcal{H}.
\end{align}
The scalars $\alpha_o$, $\beta_o$, obviously not unique,  are called \emph{lower frame bound} and \emph{upper frame bound}, respectively,  for $\{\big|\phi_{(m,n,l)}\big>\}_{m, n\in\mathbb{Z} \atop l\in\{1,2,\dots,N\}}$.
\end{defn}
\subsection{Frames for Weyl-Heisenberg Group}{\label{3a}}
The following theorem gives sufficient conditions for a sequence of functions obtained from a unitary irreducible representation associated to the  Weyl-Heisenberg group to be a Gabor frame with several generators.

\begin{thm}\label{t1}
For each $l\in\{1,2,\dots,N\}$, let $\phi_l$ be a non-zero function in $L^2(\mathbb{R},dx)$ such that support of each $\phi_l$ is contained in an interval of length $\mu_l$.
Assume that
\begin{enumerate}[$(i)$]
  \item $p_0^{(l)}=\frac{2\pi}{P\mu_l}$,
  \item $|q_0^{(l)}p_0^{(l)}|<2\pi$, that is, $|q_0^{(l)}|<|P|\mu_l$ for some fixed $q_0^{(l)}$,
\item For each $l\in\{1,2,\dots,N\}$,
\begin{align}\label{e1}
\frac{\alpha_o}{N \min\limits_{1 \leq i \leq N} \mu_i} \leq \sum_{n\in\mathbb{Z}}\big|\phi_l\big(x-nq_0^{(l)}\big)\big|^2 \leq \frac{\beta_o}{N \max\limits_{1 \leq i \leq N} \mu_i}, \ \  \text{for almost all} \ x\in \mathbb{R}.
\end{align}
\end{enumerate}
 Then, the collection  $\{\big|\phi_{(m,n,l)}^{P,Q}\big>\}_{m,n\in \mathbb{Z} \atop l\in\{1,2,\dots,N\}}$ is a multivariate frame for $L^2(\mathbb{R},dx)$ with frame  bounds  $\alpha_o$, $\beta_o$.
\end{thm}
\proof
Let $f\in L^2(\mathbb{R},dx)$ be arbitrary. By invoking the Parseval's theorem, we compute
\begin{align}{\label{h1}}
\sum_{l = 1}^{N}\sum_{m,n\in \mathbb{Z}}|\langle f|\phi_{(m,n,l)}^{P,Q}\rangle|^2
 &= \sum_{l = 1}^{N}\sum_{m,n\in \mathbb{Z}}\bigg|\int_{\mathbb{R}} \overline{f(x)}e^{i Px \frac{2\pi}{P\mu_l}m}\phi_l(x+nq_0^{(l)})dx\bigg|^2 \notag\\
&=\sum_{l = 1}^{N}\sum_{m,n\in \mathbb{Z}}\bigg|\int_{\mathbb{R}}\overline{f(y-nq_0^{(l)})}e^{i(y-nq_0^{(l)})\frac{2\pi}{\mu_l}m}\phi_l(y)dy\bigg|^2 \notag\\
&=\sum_{l = 1}^{N}\sum_{m,n\in \mathbb{Z}}\bigg|\int_{\mathbb{R}}\overline{f(y-nq_0^{(l)})}e^{i y\frac{2\pi}{\mu_l}m}\phi_l(y)dy\bigg|^2 \notag \\
&=\mu_1\sum_{n \in \mathbb{Z}}\int_{\mathbb{R}}|f(y-nq_0^{(1)})\phi_1(x)|^2dx+ \mu_2\sum_{n \in \mathbb{Z}}\int_{\mathbb{R}}|f(y-nq_0^{(2)})\phi_2(x)|^2dx \notag\\
&+\dots+\mu_N\sum_{n \in \mathbb{Z}}\int_{\mathbb{R}}|f(y-nq_0^{(N)})\phi_N(x)|^2dx \notag\\
&= \mu_1 \sum_{n \in \mathbb{Z}}\int_{\mathbb{R}}|f(x)|^2|\phi_1(x+nq_0^{(1)})|^2dx + \mu_2 \sum_{n \in \mathbb{Z}}\int_{\mathbb{R}}|f(x)|^2|\phi_2(x+nq_0^{(2)})|^2dx \notag\\
&+\dots+ \mu_N \sum_{n \in \mathbb{Z}}\int_{\mathbb{R}}|f(x)|^2|\phi_N(x+nq_0^{(N)})|^2dx \notag \\
&= \mu_1 \int_{\mathbb{R}}|f(x)|^2 \sum_{n \in \mathbb{Z}}|\phi_1(x+nq_0^{(1)})|^2dx + \mu_2 \int_{\mathbb{R}}|f(x)|^2 \sum_{n \in \mathbb{Z}}|\phi_2(x+nq_0^{(2)})|^2dx \notag \\
&+ \dots
 + \mu_N \int_{\mathbb{R}}|f(x)|^2 \sum_{n \in \mathbb{Z}}|\phi_N(x+nq_0^{(N)})|^2dx \notag \\
 &\geq \mu_1\frac{\alpha_o
 }{N \min\limits_{1 \leq i \leq N} \mu_i}\int_{\mathbb{R}}|f(x)|^2dx + \mu_2\frac{\alpha_o}{N \min\limits_{1 \leq i \leq N} \mu_i}\int_{\mathbb{R}}|f(x)|^2dx \notag\\
 & + \dots + \mu_N\frac{\alpha_o}{N \min\limits_{1 \leq i \leq N} \mu_i}\int_{\mathbb{R}}|f(x)|^2dx \notag\\
&\geq \mu_1\frac{\alpha_o}{N \min\limits_{1 \leq i \leq N} \mu_i}\|f\|^2 + \mu_2\frac{\alpha_o}{N \min\limits_{1 \leq i \leq N} \mu_i}\|f\|^2 + \dots +
\mu_N\frac{\alpha_o}{N \min\limits_{1 \leq i \leq N} \mu_i}\|f\|^2 \notag\\
&\geq \min\limits_{1 \leq i \leq N} \mu_i \frac{\alpha_o}{N \min\limits_{1 \leq i \leq N} \mu_i}\|f\|^2 + \min\limits_{1 \leq i \leq N} \mu_i\frac{\alpha_o}{N \min\limits_{1 \leq i \leq N} \mu_i}\|f\|^2 \notag\\
&+ \dots +\min\limits_{1 \leq i \leq N} \mu_i \frac{\alpha_o}{N \min\limits_{1 \leq i \leq N} \mu_i}\|f\|^2 \notag \\
&= \alpha_o \|f\|^2.
\end{align}
Similarly, for any  $f\in L^2(\mathbb{R},dx)$, we have
\begin{align}{\label{h2}}
&\sum_{l = 1}^{N}\sum_{m,n\in \mathbb{Z}}|\langle f|\phi_{(m,n,l)}^{P,Q}\rangle|^2 \notag\\
& \leq \mu_1\frac{\beta_o}{N \max\limits_{1 \leq i \leq N} \mu_i}\int_{\mathbb{R}}|f(x)|^2dx + \mu_2\frac{\beta_o}{N \max\limits_{1 \leq i \leq N} \mu_i}\int_{\mathbb{R}}|f(x)|^2dx + \dots +
\mu_N\frac{\beta_o}{N \max\limits_{1 \leq i \leq N} \mu_i}\int_{\mathbb{R}}|f(x)|^2dx \notag\\
&\leq \mu_1\frac{\beta_o}{N \max\limits_{1 \leq i \leq N} \mu_i} \|f\|^2 + \mu_2\frac{\beta_o}{N \max\limits_{1 \leq i \leq N} \mu_i} \|f\|^2+\dots+ \mu_N\frac{\beta_o}{N \max\limits_{1 \leq i \leq N} \mu_i} \|f\|^2 \notag\\
& \leq \max\limits_{1 \leq i \leq N} \mu_i \frac{\beta_o}{N \max\limits_{1 \leq i \leq N} \mu_i} \|f\|^2 + \max\limits_{1 \leq i \leq N} \mu_i \frac{\beta_o}{N \max\limits_{1 \leq i \leq N} \mu_i} \|f\|^2+\dots+ \max\limits_{1 \leq i \leq N} \mu_i \frac{\beta_o}{N \max\limits_{1 \leq i \leq N} \mu_i} \|f\|^2 \notag \\
&= \beta_o \|f\|^2.
\end{align}
Inequalities $(\ref{h1})$ and $(\ref{h2})$  concludes the result.
\endproof

\begin{rem}
For $N > 1$, the converse of Theorem $\ref{t1}$ is not true. This is justified in Example \ref{eg1}, where following result is used.
\end{rem}
\begin{thm}\cite[Theorem 11.4.2]{ole}\label{t2}
Let $a$, $b>0$ and $g$ be a non-zero function in  $L^2(\mathbb{R},dx)$. Suppose that
\begin{align*}
& \beta_o:=\frac{1}{b}\sup_{x\in[0,a]}\sum_{k\in \mathbb{Z}}\bigg|\sum_{n\in \mathbb{Z}}g(x-na)\overline{g(x-na-k/b)}\bigg|<\infty,
\intertext{and}
&\gamma_o:=\frac{1}{b}\inf_{x\in[0,a]}\biggr[\sum_{n\in \mathbb{Z}}|g(x-na)|^2- \sum_{k \neq 0}\bigg|\sum_{n \in \mathbb{Z}}g(x-na)\overline{g(x-na-k/b)}\bigg|\biggr] > 0.
\end{align*}
Then, $\{e^{2\pi imbx}g(x-na)\}_{m,n\in \mathbb{Z}}$ is a frame for $ L^2(\mathbb{R},dx)$ with frame  bounds $\gamma_o$, $\beta_o$.
\end{thm}
Now, we show that, for $N > 1$,  the converse of Theorem $\ref{t1}$ is not true.
\begin{exa}\label{eg1}
Let $P=1$, $Q=0$ and define a function $\phi_1 \in L^2(\mathbb{R},dx)$, with $\mu_1=2$, as follows:
\begin{align*}
\phi_1(x)=\begin{cases}
1+x, \quad & \text{if}\ x \in (0,1]; \\
x, \quad & \text{if}\ x \in (1,2];\\
0, \quad & \text{elsewhere}.
\end{cases}
\end{align*}
Choose $q_0^{(1)}=1$, $a=1$ and $b=\frac{1}{\mu_1}=\frac{1}{2}$. Consider for $n$, $k\in \mathbb{Z}$, the function $x \rightarrow \phi_1(x-n)\phi_1(x-n-2k)$ for $x\in (0,1]$. Due to the compact support of $\phi_1$, it can only be non-zero if $n \in \{-1,0\}$ and $k=0$. So,\\
\begin{align*}
& G_0(x)=\sum_{n \in \mathbb{Z}}|g(x-n)|^2=|g(x)|^2+|g(x+1)|^2=2(x^2+2x+1) \ \text{for}\ x\in (0,1],\\
\intertext{and}
& G_1(x)=\sum_{k \neq 0}\bigg|\sum_{n \in \mathbb{Z}}g(x-n)\overline{g(x-na-2k)}\bigg|=0 \ \text{for} \ x \in (0,1].
\end{align*}
Therefore, by Theorem \ref{t2}, $\{e^{2 ix\frac{m}{\mu_1}\pi}\phi_1(x+n)\}_{m,n\in \mathbb{Z}}$ is a frame for $L^2(\mathbb{R},dx)$ with frame bounds $4$ and $16$. Hence, $\big\{\big|\phi_{(m,n,1)}^{1,0}\big>\big\}_{m,n\in \mathbb{Z}}$ is a frame for $L^2(\mathbb{R},dx)$ with frame bounds $4$ and $16$.

Similarly, for the following function  $\phi_2$ in $L^2(\mathbb{R},dx)$,
\begin{align*}
\phi_2(x)=\begin{cases}
1+x, \quad & \text{if}\ x \in (0,1], \\
\frac{x}{2}, \quad & \text{if}\ x \in (1,2],\\
0, \quad & \text{elsewhere}.
\end{cases}
\end{align*}
 with $q_0^{(2)}=1, a=1$ and $b=\frac{1}{\mu_2}=\frac{1}{2}$, the sequence $\big\{\big|\phi_{(m,n,2)}^{1,0}\big>\big\}_{m,n\in \mathbb{Z}}$  is a frame for $L^2(\mathbb{R},dx)$ with frame bounds $\frac{1}{4}$ and $9$.
Thus, for all $f  \in L^2(\mathbb{R},dx)$, we have
\begin{align*}
\frac{17}{4}\|f\|^2 \leq \sum_{l =1}^{2}\sum_{m,n\in \mathbb{Z}}|\langle f|\phi_{(m,n,l)}^{1,0}\rangle|^2 \leq 25\|f\|^2.
\end{align*}
Hence,  the sequence $\{\big|\phi_{(m,n,l)}^{1,0}\big>\}_{m,n\in \mathbb{Z},l\in\{1,2\}}$ is a multivariate frame with bound $\frac{17}{4}$ and $25$.

One may observe that  for $x=\sqrt{3.5}$, we have $\sum_{n\in \mathbb{Z}}|\phi_1(x-n)|^2=7$. But this is not possible as inequality (\ref{e1}) implies that
\begin{align*}
\frac{17}{16} \leq \sum\limits_{n\in \mathbb{Z}}|\phi_l(x-n)|^2 \leq \frac{25}{4} \ \ \text{for} \ l\in \{1,2\}, \ \text{for almost all} \ x\in \mathbb{R}.
\end{align*}
\end{exa}

\begin{rem}
The converse of Theorem \ref{t1} is true if inequality (\ref{e1}) is replaced by
\begin{align*}
\alpha_o \leq \sum_{l=1}^{N} \sum_{n\in \mathbb{Z}} \mu_l \big|\phi_l\big(x-nq_0^{(l)}\big)\big|^2 \leq \beta_o, \ \text{for almost all} \   x\in \mathbb{R}.
\end{align*}
\end{rem}
\subsection{Frames for Extended Affine Group}
In this section, we give existence of wavelet frames with several generators  associated with  extended affine group. Here, we fix a representation $\kappa^{P, Q}$ of the Weyl-Heisenberg group on $L^2(\mathbb{R})$; and let  $u(\epsilon)=u_0+\frac{P}{\epsilon}$ and $v(\epsilon)= v_0-\frac{P}{\epsilon}$ be such that $u_0 +v_0=Q$.  Define a sequence on extended affine group, $EA$, by
\begin{align*}
\big|\phi_{(m,n,l)}^{\epsilon, u(\epsilon), v(\epsilon)}\big>= \zeta_\epsilon ^{u(\epsilon), v(\epsilon)}(a,b,y)\phi_l,
\end{align*}
where $m,n \in \mathbb{Z}$ and $l \in \{1,2,\dots, N\}$ and the  symbol $\zeta_\epsilon ^{\bullet, \bullet}(.,.,.)$ is defined in Section \ref{sec2.3}. The following result gives multivariate wavelet frames for signal space $L^2(\mathbb{R},dx)$.

\begin{thm}\label{t3}
Under the assumptions of Theorem \ref{t1}, for any $\epsilon \in (0,1]$, let $EA_{q_0(l), p_0(l)}(\epsilon)$ be a  discrete subset of $EA$ given by
\begin{align*}
    \Big\{Q_n^{(l)}(\epsilon), b_{mn}^{(l)}(\epsilon), y_{mn}^{(l)}(\epsilon) | \  m, n \in \mathbb{Z}, l \in \{1,2,\dots, N\}\Big\},
  \end{align*}
where
\begin{align*}
a_n^{(l)}(\epsilon)&= e^{-\epsilon n q_0^{(l)}},\\
b_{mn}^{(l)}(\epsilon)&=\frac{\pi a_n^{(l)}(\epsilon)m \beta_o }{\alpha_o v(\epsilon)(ln(\epsilon+c)-ln c)c\mu_l}, \\
&=\frac{\beta_o e^{-\epsilon n q_0^{(l)}}\pi m  }{\alpha_o v(\epsilon)(ln(\epsilon+c)-ln c)c\mu_l} \\
y_{mn}^{(l)}(\epsilon)&=b_{mn}(\epsilon)
\frac{l_na_n^{(l)}(\epsilon)}{a_n^{(l)}(\epsilon)-1}.
\end{align*}
Then,  the collection  $\big\{\big|\phi^{\epsilon, u(\epsilon), v(\epsilon)}_{(m,n,l)}\big >\big\}_{m, n \in \mathbb{Z} \atop l \in \{1,2,\dots, N\}} = \Big\{\zeta_{\epsilon}^{u(\epsilon),
v(\epsilon)}(a_{n}^{(l)}(\epsilon), \beta_{mn}^{(l)}(\epsilon), \gamma_{mn}^{(l)}(\epsilon))(T_{\epsilon}\phi)\Big\}_{m, n \in \mathbb{Z} \atop l \in \{1,2,\dots, N\}}$,
where $T_{\epsilon}(x)=e^{\frac{-1}{2}\epsilon x}$, forms a  frame for $L^2(\mathbb{R}, dx)$ with frame bounds
\begin{align*}
\frac{\alpha_o^2 c}{\beta_o} \Big(\frac{ln(\epsilon +c) - ln c}{\epsilon}\Big) \ \quad   \text{and} \ \quad  \alpha_o c \Big(\frac{ln(\epsilon +c) - ln c}{\epsilon}\Big).
\end{align*}
\end{thm}
\proof
For any  $f \in L^2(\mathbb{R},dx)$, we compute
\begin{align*}
&\sum_{l = 1}^{N}
\underset{ m, n  \in \mathbb{Z}}{\sum} \bigg|\big \langle f| \phi^{\epsilon, u(\epsilon), v(\epsilon)}_{(m,n,l)} \big \rangle\bigg|^2  \\
&=\sum_{l = 1}^{N}
\underset{ m, n  \in \mathbb{Z}}{\sum}\bigg|\int_\mathbb{R}\overline{f(x)} e^{i v(\epsilon)b_{mn}^{(l)}(\epsilon)e^{-\epsilon x}}e^{\frac{-\epsilon}{2}(x-\frac{lna_n^{(l)}(\epsilon)}{\epsilon})} \phi_l(x-\frac{ln a_n^{(l)}(\epsilon)}{\epsilon})dx\bigg|^2  \\
&=\sum_{l = 1}^{N}
\underset{ m, n  \in \mathbb{Z}}{\sum}\bigg|\int_\mathbb{R}\overline{f(x)} e^{i v(\epsilon)b_{mn}^{(l)}(\epsilon)e^{-\epsilon x}}e^{\frac{-\epsilon}{2}(x+nq_0^{(l)})}\phi_l(x+nq_0^{(l)})dx\bigg|^2   \\
 &=\sum_{l = 1}^{N}
 \underset{ m, n  \in \mathbb{Z}}{\sum}\bigg|\int_{\mathbb{R}^+}\overline{f(\frac{-\ln{s}}{\epsilon})} e^{i v(\epsilon)b_{mn}^{(l)}(\epsilon)s}e^{\frac{-\epsilon}{2}(\frac{-\ln{s}}{\epsilon}+nq_0^{(l)})}\phi_l(\frac{-\ln{s}}{\epsilon}+nq_0^{(l)})\frac{ds}{\epsilon s}\bigg|^2  \\
 &=\sum_{l = 1}^{N}
 \underset{ m, n  \in \mathbb{Z}}{\sum}\bigg|\int_{\mathbb{R}^+}\overline{f(\frac{-\ln{s}}{\epsilon})} e^{i v(\epsilon)b_{mn}^{(l)}(\epsilon)s}e^{\frac{-\epsilon}{2}(nq_0^{(l)})}\phi_l(\frac{-\ln{s}}{\epsilon}+nq_0^{(l)})\frac{ds}{\epsilon \sqrt{s}}\bigg|^2   \\
 &=\sum_{l = 1}^{N}
 \underset{ m, n  \in \mathbb{Z}}{\sum}\bigg|\int_{E_l}\overline{f(\frac{-\ln{t(a_n^{(l)}(\epsilon))}^{-1}}{\epsilon})} e^{i v(\epsilon)b_{mn}^{(l)}(\epsilon)\frac{t}{a_n^{(l)}(\epsilon)}}e^{\frac{-\epsilon}{2}(nq_0^{(l)})}\phi_l(\frac{-\ln{t}}{\epsilon})\frac{dt}{\epsilon \sqrt{a_n(\epsilon)t}}\bigg|^2 \nonumber\\
 &=\sum_{l = 1}^{N}
 \underset{ m, n  \in \mathbb{Z}}{\sum}\bigg|\int_{E_l}\overline{f(\frac{-\ln{t(a_n^{(l)}(\epsilon))}^{-1}}{\epsilon})} e^{\frac{iB_o\pi m}{A_o(\ln{\epsilon+c}- \ln{c})\mu_lc}}e^{\frac{-\epsilon}{2}(nq_0^{(l)})}\phi_l(\frac{-\ln{t}}{\epsilon})\frac{dt}{\epsilon \sqrt{a_n(\epsilon)t}}\bigg|^2\\
 & = (\star \star),
 \end{align*}
 where $E_l \subseteq \mathbb{R}$ with finite measure depending on $\mu_l$ $(l \in \{1,2, \dots, N\})$.

By invoking  the Parseval's theorem, we get that
 \begin{align}\label{Eq3.6II}
(\star \star)
 &=\sum_{l = 1}^{N} \frac{\alpha_oc\mu_l(\ln{\epsilon+c}-\ln{c})}{B_o}
 \underset{ n \in \mathbb{Z}}{\sum}\int_{E_l}\bigg|\overline{f(\frac{-\ln{t(a_n^{(l)}(\epsilon))}^{-1}}{\epsilon})} e^{\frac{-\epsilon}{2}(nq_0^{(l)})}\phi_l(\frac{-\ln{t}}{\epsilon})\frac{1}{\epsilon \sqrt{a_n(\epsilon)t}}\bigg|^2dt \nonumber\\
 &=\sum \sum_{l = 1}^{N} \frac{\alpha_oc\mu_l(\ln{\epsilon+c}-\ln{c})}{\beta_o\epsilon}
 \underset{ n \in \mathbb{Z}}{\sum}\int_{\mathbb{R}^+}\bigg|\overline{f(\frac{-\ln{t(a_n^{(l)}(\epsilon))}^{-1}}{\epsilon})} e^{\frac{-\epsilon}{2}(nq_0^{(l)})}\phi_l(\frac{-\ln{t}}{\epsilon})\bigg|^2\frac{1}{\epsilon t}dt.
\end{align}
Change of variables $s=\frac{t}{a_n^{(l)}(\epsilon)}$ and then $x=\frac{-\ln{s}}{\epsilon}$, in Eq. \eqref{Eq3.6II},  gives
\begin{align}\label{Eq3.7II}
&\sum_{l = 1}^{N}
\underset{ m, n  \in \mathbb{Z}}{\sum} \big|\big \langle f| \phi^{\epsilon, u(\epsilon), v(\epsilon)}_{(m,n,l)} \big \rangle \big|^2 \nonumber\\
&= \sum_{l = 1}^{N} \frac{\alpha_o(\ln{(\epsilon+c})-\ln{c})c\mu_l}{\beta_o\epsilon}\underset{ n \in \mathbb{Z}}{\sum}\int_{\mathbb{R}^+}\bigg|\overline{f(\frac{-\ln{s}}{\epsilon})} \phi_l(\frac{-\ln{s}a_n^{(l)}(\epsilon)}{\epsilon})\bigg|^2\frac{1}{\epsilon s}ds
 \nonumber \\
 &= \sum_{l = 1}^{N} \frac{\alpha_o(\ln{(\epsilon+c})-\ln{c})c\mu_l}{B_o\epsilon}\underset{ n \in \mathbb{Z}}{\sum}\int_{\mathbb{R}^+}\big|\overline{f(x)} \phi_l(x+nq_0^{(l)})\big|^2dx.
\end{align}
Using Eq. \eqref{Eq3.7II}, we have
    \begin{align*}
\frac{\alpha_o^2(ln{(\epsilon+c)}-\ln{c})c||f||^2}{\beta_o\epsilon} \leq \sum_{ l \in \{1,2, \dots, N\}}\underset{ m, n  \in \mathbb{Z}}{\sum} \bigg|\big \langle f| \phi^{\epsilon, u(\epsilon), v(\epsilon)}_{(m,n,l)} \big \rangle \bigg|^2
\leq \frac{\alpha_o(ln{(\epsilon+c)}-\ln{c})c||f||^2}{\epsilon},
\end{align*}
for all  $f \in L^2(\mathbb{R},dx)$. Hence, $\{\phi^{\epsilon, u(\epsilon), v(\epsilon)}_{(m,n,l)} \}_{m, n \in \mathbb{Z}, l \in \{1,2,\dots, N\} }$ is a  frame for $L^2(\mathbb{R},dx)$ with the desired frame bounds.
\endproof

Next theorem gives us the relation between  frames with several generators for Weyl-Heisenberg group and extended affine group.

\begin{thm} \label{t4}
Under the assumptions of Theorem \ref{t1}, for any $f \in L^2(\mathbb{R}, dx)$, we have
\begin{align*}
\lim_{\epsilon \to 0^+ }\sum_{l = 1}^{N}
 \underset{ m, n  \in \mathbb{Z}}{\sum} |\langle f, \phi^{\epsilon, u(\epsilon), v(\epsilon)}_{(m,n,l)} \rangle|^2 \leq \sum_{l = 1}^{N}
 \underset{ m, n  \in \mathbb{Z}}{\sum} |\langle f, \phi^{P, Q}_{(m,n,l)} \rangle|^2 \leq \beta_o||f||^2
\end{align*}
\end{thm}
\proof
For any  $f \in L^2(\mathbb{R}, dx)$,  we have
\begin{align*}
\lim_{\epsilon \to 0^+ }\sum_{l = 1}^{N}
 \underset{ m, n  \in \mathbb{Z}}{\sum} |\langle f, \phi^{\epsilon, u(\epsilon), v(\epsilon)}_{(m,n,l)} \rangle|^2
 &\leq \lim_{\epsilon \to 0^+ }\frac{\alpha_o(\ln{(\epsilon+c)}-\ln{c})c||f||^2}{\epsilon}\\
 &= \lim_{\epsilon \to 0^+ }\frac{\alpha_o c}{\epsilon + c}||f||^2\\
 &=\alpha_o||f||^2\\
 &\leq \sum_{l = 1}^{N}
 \underset{ m, n  \in \mathbb{Z}}{\sum} |\langle f, \phi^{P, Q}_{(m,n,l)} \rangle|^2\\
 &\leq \beta_o||f||^2.
\end{align*}
This concludes the proof.
\endproof

The following example illustrates Theorem \ref{t4}.
\begin{exa}
Consider the Weyl-Heisenberg frame $\big\{\big|\phi_{(m,n,1)}^{1,0}\big>\big\}_{m,n\in \mathbb{Z}}$  for $L^2(\mathbb{R},dx)$ with frame bounds $4$ and $16$ given in Example \ref{eg1}.
Let $u_0=-1, v_0=1$ and $c=1$ be a fixed constant.

Define
\begin{align*}
a_n^{(1)}(\epsilon)&= e^{-\epsilon n};\\
b_{mn}^{(1)}(\epsilon)&=\frac{16 \pi a_n^{(1)}(\epsilon)m}{4v(\epsilon)(ln(\epsilon+1)-ln 1)\mu_1}
=\frac{2e^{-\epsilon n}\pi m  }{(1-\frac{1}{\epsilon})(ln(\epsilon+1)-ln 1)}; \\
y_{mn}^{(1)}(\epsilon)&=b_{mn}(\epsilon)
\frac{l_na_n^{(1)}(\epsilon)}{a_n^{(1)}(\epsilon)-1}.
\end{align*}
Then, it is easy to see that $2 \leq \sum_{n\in \mathbb{Z}}|\phi_l(x-n)|^2 \leq 8$ for almost all $x\in \mathbb{R}$.
Hence, by Theorem \ref{t3}, the collection  $\big\{\big|\phi^{\epsilon, u(\epsilon), v(\epsilon)}_{(m,n,1)}\big >\big\}_{m, n \in \mathbb{Z}}$ is a frame with frame bounds $ \frac{ln(\epsilon+1)}{\epsilon}$ and $ \frac{4ln(\epsilon+1)}{\epsilon}$.

Now, for any $f \in L^2(\mathbb{R}, dx)$, we have
\begin{align*}
\lim_{\epsilon \to 0^+ }\sum_{ m, n  \in \mathbb{Z}} |\langle f, \phi^{\epsilon, u(\epsilon), v(\epsilon)}_{(m,n,1)} \rangle|^2
 &\leq \lim_{\epsilon \to 0^+ } \frac{4ln(\epsilon+1)}{\epsilon}\|f\|^2\\
 & \leq 4 \|f\|^2 \\
&\leq \sum_{m, n  \in \mathbb{Z}} |\langle f, \phi^{1,0}_{(m,n,1)} \rangle|^2\\
 &\leq 16||f||^2.
\end{align*}
\end{exa}

\section{A Perturbation Result}\label{sec4}
The importance of  perturbation theory is well-known and has been extensively studied, see \cite{Kato}. For a variety of perturbation results for different types of frames, we refer to \cite{CK,  FZUS, Heil20}. In this section, we give Paley-Wiener type perturbation results for multivariate frames for $L^2(\mathbb{R}, dx)$.  The following result shows that frames for $L^2(\mathbb{R}, dx)$ associated with the Weyl-Heisenberg group   are stable under small perturbation.
\begin{thm}\label{pthm1}
 Let $\big\{\big|\phi^{P, Q}_{(m,n,l)}\big> \big\}_{m, n \in \mathbb{Z} \atop l \in \{1,2,\dots, N\} }$ be a  frame for $L^2(\mathbb{R}, dx)$ with frame bounds $\alpha_o$ and $\beta_o$ and let $\big\{\big|\Tilde{\phi}^{P,Q}_{(m,n,l)}\big> \big\}_{m, n \in \mathbb{Z} \atop l \in \{1,2,\dots, N\} } \subset  L^2(\mathbb{R}, dx)$.  For non-negative real number $\lambda$ and $M$ with  $\lambda <\alpha_o$, assume that
\begin{align}\label{p1}
& \sum_{l = 1}^{N}\sum_{m,n\in\mathbb{Z}} |\langle f| \phi^{P, Q}_{(m,n,l)}- \Tilde{\phi}^{P, Q}_{(m,n,l)} \rangle|^2 \nonumber \\
&\leq M \min \Big\{\sum_{l = 1}^{N} \sum_{m,n\in\mathbb{Z}} |\langle f| \phi^{P, Q}_{(m,n,l)}\rangle|^2, \  \sum_{l = 1}^{N} \sum_{m,n\in\mathbb{Z}} |\langle f| \Tilde{\phi}^{P, Q}_{(m,n,l)} \rangle|^2
\Big\} + \lambda ||f||^2,
\end{align}
  for all  $f \in L^2(\mathbb{R}, dx)$.  Then, $\big\{\big|\Tilde{\phi}^{P,Q}_{(m,n,l)}\big> \big\}_{m, n \in \mathbb{Z} \atop  l \in \{1,2,\dots, N\} }$ is a frame for $L^2(\mathbb{R}, dx)$ with frame bounds
\begin{align*}
\frac{1}{2(M+1)}\Big( 1-\frac{\lambda}{\alpha_o} \Big)\alpha_o \quad  \text{and} \quad  \Big( 2\beta_o(M+1)+\lambda \Big).
 \end{align*}
\end{thm}
\begin{proof}
By hypothesis \eqref{p1}, for any $f \in L^2(\mathbb{R}, dx)$, we have
\begin{align}\label{p2}
& \sum_{l = 1}^{N}\sum_{m,n\in\mathbb{Z}} |\langle f| \Tilde{\phi}^{P, Q}_{(m,n,l)}\rangle|^2 \nonumber\\
& \leq \sum_{l = 1}^{N} \sum_{m,n\in\mathbb{Z}} |\langle f| \phi^{P, Q}_{(m,n,l)}- \Tilde{\phi}^{P, Q}_{(m,n,l)} \rangle|^2 + 2\sum_{l = 1}^{N} \sum_{m,n\in\mathbb{Z}} |\langle f| \Tilde{\phi}^{P, Q}_{(m,n,l)}\rangle|^2 \nonumber\\
& \leq 2M \min \Big(\sum_{l = 1}^{N} \sum_{m,n\in\mathbb{Z}} |\langle f| \phi^{P, Q}_{(m,n,l)}\rangle|^2,  \sum_{l = 1}^{N} \sum_{m,n\in\mathbb{Z}} |\langle f| \Tilde{\phi}^{P, Q}_{(m,n,l)} \rangle|^2 \Big)+\lambda ||f||^2  \nonumber\\
&+ 2\sum_{l = 1}^{N} \sum_{m,n\in\mathbb{Z}} |\langle f| \Tilde{\phi}^{P, Q}_{(m,n,l)}\rangle|^2  \nonumber \\
& \leq (2M+2)\sum_{l = 1}^{N}\sum_{m,n\in\mathbb{Z}} |\langle f| \phi^{P, Q}_{(m,n,l)}\rangle|^2 +\lambda||f||^2 \nonumber \\
& \leq 2(M+1)\beta_o||f||^2 +\lambda||f||^2 \nonumber \\
&=\big( 2\beta_o(M+1)+\lambda \big)||f||^2.
\end{align}
Similarly, using \eqref{p1}, for any $f\in L^2(\mathbb{R}, dx)$, we can show that
 \begin{align*}
     \sum_{l = 1}^{N}\sum_{m,n\in\mathbb{Z}} |\langle f| \phi^{P, Q}_{(m,n,l)}\rangle|^2 - \lambda||f||^2 \leq (2M+2)\sum_{l = 1}^{N}\sum_{m,n\in\mathbb{Z}} |\langle f| \Tilde{\phi}^{P, Q}_{(m,n,l)}\rangle|^2,
 \end{align*}
which entails
\begin{align*}
    \sum_{l = 1}^{N} \sum_{m,n\in\mathbb{Z}} |\langle f| \phi^{P, Q}_{(m,n,l)}\rangle|^2 - \frac{\lambda}{\alpha_o} \sum_{l = 1}^{N} \sum_{m,n\in\mathbb{Z}} |\langle f| \phi^{P, Q}_{(m,n,l)}\rangle|^2\leq 2(M+1)\sum_{l = 1}^{N} \sum_{m,n\in\mathbb{Z}} |\langle f| \Tilde{\phi}^{P, Q}_{(m,n,l)}\rangle|^2.
\end{align*}
That is
\begin{align}\label{p3}
    \frac{1}{2(M+1)}\big( 1-\frac{\lambda}{\alpha_o} \big)\alpha_o||f||^2 \leq \sum_{l = 1}^{N} \sum_{m,n\in\mathbb{Z}} |\langle f| \Tilde{\phi}^{P, Q}_{(m,n,l)}\rangle|^2 \ \text{for all} \ f \in L^2(\mathbb{R}, dx).
\end{align}
Hence, conclusion follows from \eqref{p2} and \eqref{p3}.
\end{proof}

 \begin{rem}
Theorem \ref{pthm1} gives a perturbation result for multivariate Gabor frame such that the new sequence  $\big\{\big|\Tilde{\phi}^{P,Q}_{(m,n,l)}\big> \big\}_{m, n \in \mathbb{Z} \atop l \in \{1,2,\dots, N\} }$ is a  Gabor frame. Using Theorem \ref{t4}, this new  Gabor frame can be related to some  wavelet frame $\{\big|\Tilde{\phi}^{\epsilon, u(\epsilon), v(\epsilon)}_{(m,n,l)}\big> \}_{m, n \in \mathbb{Z} \atop l \in \{1,2,\dots, N\} }$ associated with the extended affine group which can be constructed using perturbation result. More precisely,  Let $\{\big|\phi^{\epsilon, u(\epsilon), v(\epsilon)}_{(m,n,l)}\big> \}_{m, n \in \mathbb{Z} \atop l \in \{1,2,\dots, N\} }$ be a  frame for $L^2(\mathbb{R}, dx)$ with frame bounds
\begin{align*}
\frac{\alpha_o^2 c}{\beta_o\epsilon}\big(\ln(\epsilon+c)-\ln c\big) \ \text{and} \  \alpha_o c\big(\frac{\ln(\epsilon+c)-\ln c}{\epsilon}\big),
\end{align*}
and let $\{\big|\Tilde{\phi}^{\epsilon, u(\epsilon), v(\epsilon)}_{(m,n,l)}\big> \}_{m, n \in \mathbb{Z} \atop l \in \{1,2,\dots, N\} } \subset  L^2(\mathbb{R}, dx)$. Let  $\lambda$, $M \geq 1$ be such that $\lambda <\frac{\alpha_o^2 c}{\beta_o\epsilon}\big(\ln(\epsilon+c)-\ln c\big)$.  For any $f \in L^2(\mathbb{R}, dx)$,  assume that
 \begin{align*}
 & \sum_{l = 1}^{N} \sum_{m,n\in\mathbb{Z}} |\langle f| \phi^{\epsilon, u(\epsilon), v(\epsilon)}_{(m,n,l)}- \Tilde{\phi}^{\epsilon, u(\epsilon), v(\epsilon)}_{(m,n,l)} \rangle|^2
 \\
 &\leq M \min \Big\{ \sum_{l = 1}^{N} \sum_{m,n\in\mathbb{Z}} |\langle f| \phi^{\epsilon, u(\epsilon), v(\epsilon)}_{(m,n,l)}\rangle|^2, \  \sum_{l = 1}^{N}\sum_{m,n\in\mathbb{Z}} |\langle f| \Tilde{\phi}^{\epsilon, u(\epsilon), v(\epsilon)}_{(m,n,l)} \rangle|^2 \Big\} + \lambda \|f\|^2.
 \end{align*}
 Then,  $\{\big|\Tilde{\phi}^{\epsilon, u(\epsilon), v(\epsilon)}_{(m,n,l)}\big> \}_{m, n \in \mathbb{Z} \atop  l \in \{1,2,\dots, N\} }$ is a  frame for $L^2(\mathbb{R}, dx)$ with frame bounds
 \begin{align*}
  \frac{1}{2(M+1)}\Big( 1-\frac{\lambda \beta\epsilon}{\alpha^2 c(\ln(\epsilon+c)-\ln c)} \Big)\Big(\frac{\alpha_o^2 c}{\beta_o\epsilon}(\ln(\epsilon+c)-\ln c)\Big) \  \intertext{and} \
   \big(  \frac{2\alpha_o c}{\epsilon}(\ln(\epsilon+c)-\ln c)(M+1)+\lambda \big).
  \end{align*}
\end{rem}
\section{Conclusion}
Gabor frames has potential applications in time-frequency analysis and signal analysis. By using fundamental tools developed in \cite{SB} for the construction of Parseval frames and using a unitary irreducible representation associated to the Weyl-Heisenberg group, we give frame conditions for  Gabor systems with several generators in the signal space $L^2(\mathbb{R}, dx)$.   Afterwards, we give frame conditions for  wavelet systems with several generators associated with an extended affine group. It would be interesting to know  that: How multivariate frames for the Weyl-Heisenberg group and the extended affine group are related? We provide a relation between frame bounds of multivariate Gabor systems and wavelet systems in the signal space $L^2(\mathbb{R}, dx)$. Stability of frames is equally important in study of stable signal analysis, in particular,  stable reconstruction and decomposition of signals.  A stability result, with explicit frame bounds, for  Gabor frames of the signal space  $L^2(\mathbb{R}, dx)$ is discussed. These results are useful in applications of frames.

$$\textbf{\text{Data Related Statement}}$$
No data used in this study.

$$\textbf{\text{Conflicts of Interest}}$$
The authors have no conflicts of interest.

\end{document}